\documentclass{article}

\usepackage{amscd,amsmath,amssymb,amsthm,stackrel}
\usepackage{color}
\newcommand{\eps}{\varepsilon}

\newcommand{\R}{\mathbb R}
\newcommand{\N}{\mathbb N}

\newcommand{\then}{\Longrightarrow}
\newcommand{\J}{{\cal J}}
\newcommand{\A}{{\cal A}}

\newcommand{\arctg}{{\rm arctg}}

\DeclareMathOperator*{\esssup}{ess\; sup}

\newcommand\meas{{\rm meas}}

\newtheorem{corollary}{Corollary}[section]
\newtheorem{theorem}[corollary]{Theorem}
\newtheorem{lemma}[corollary]{Lemma}
\newtheorem{proposition}[corollary]{Proposition}

\theoremstyle{definition}
\newtheorem{definition}[corollary]{Definition}
\newtheorem{remark}[corollary]{Remark}
\newtheorem{example}[corollary]{Example}

\numberwithin{equation}{section}
\begin{document}

\title{{\bf Quasilinear problems without the Ambrosetti--Rabinowitz condition}}

\author{{\bf\large Anna Maria Candela$^{1}$\footnote{Member of the {\sl Gruppo Nazionale per
l'Analisi Matematica, la Probabilit\`a e le loro Applicazioni} (GNAMPA)
of the {\sl Istituto Nazionale di Alta Matematica} (INdAM)}, Genni Fragnelli$^{2,*}$,
Dimitri Mugnai$^{3,*}$}\\
{\small $^{1,2}$Dipartimento di Matematica}\\
{\small Universit\`a degli Studi di Bari Aldo Moro} \\
{\small Via E. Orabona 4, 70125 Bari, Italy}\\
{\small \it $^1$annamaria.candela@uniba.it, $^2$genni.fragnelli@uniba.it}
\vspace{1mm}\\
{\small $^3$Dipartimento di Scienze Ecologiche e Biologiche}\\
{\small Universit\`a degli Studi della Tuscia}\\
{\small Largo dell'Universit\`a, 01100 Viterbo, Italy}\\
{\small \it dimitri.mugnai@unitus.it}
\vspace{1mm}}
\date{}

\maketitle

\begin{abstract}
We show the existence of nontrivial solutions for a class of
quasilinear problems in which the governing operators 
depend on the unknown function. By using a suitable variational 
setting and a weak version of the Cerami--Palais--Smale condition, 
we establish the desired result without assuming that the nonlinear 
source satisfies the Ambrosetti--Rabinowitz condition.
\end{abstract}

\noindent
{\it \footnotesize 2010 Mathematics Subject Classification}: {\scriptsize 35J92, 35J20, 35J60}.\\
{\it \footnotesize Key words}: {\scriptsize Quasilinear equation, weak Cerami--Palais--Smale condition, 
failure of the Ambrosetti--Rabinowitz condition, $p$--superlinear problem, subcritical growth}.


\section{Introduction} \label{secintroduction}

In this paper we investigate the existence of weak bounded solutions of the  problem 
\begin{equation}\label{euler}
\left\{
\begin{array}{lr}
- {\rm div} (A(x,u) |\nabla u|^{p-2} \nabla u) + \frac1p\ A_t(x,u)
 |\nabla u|^p \  =\  g(x,u) & \hbox{in $\Omega$,}\\
u\ = \ 0 & \hbox{on $\partial\Omega$,}
\end{array}
\right.
\end{equation}
with $\Omega \subset \R^N$ bounded domain, $N\ge 1$,
$p > 1$, where $A(x,t)$, $g(x,t)$ are
given real functions on $\Omega \times \R$ and $A_t(x,t) = \frac{\partial}{\partial t} A(x,t)$.

Due to the fact that the divergence term depends also on the unknown function $u$,
the given equation is  quasilinear and cannot be studied with standard variational techniques.
For this reason, in the last years different approaches have been developed 
involving nonsmooth tools (see \cite{Ca,CD,CDM}) or a suitable definition of critical point, since the 
weak solutions of \eqref{euler} require as test functions 
only elements of $W^{1,p}_0(\Omega)$ which are also in $L^\infty(\Omega)$
(see \cite{AB1}). 

More recently, the idea has been to set
problem \eqref{euler} in a suitable Banach space $X$, namely 
$X= W^{1,p}_0(\Omega)\cap L^\infty(\Omega)$ equipped with the intersection norm
$\|\cdot\|_X = \|\cdot\|_W + |\cdot|_\infty$, so that
its weak solutions coincide with the {\sl true} critical points of the associated functional
\begin{equation}
\label{funct}
\J(u)\ =\ \frac1p \int_\Omega A(x,u)|\nabla u|^p dx - \int_\Omega G(x,u) dx,\qquad u\in X,
\end{equation}
with $G(x,t)= \int_0^tg(x,\tau)d\tau$ (see \cite{CP1,CP2}).

Following such an approach, in this paper we consider suitable assumptions,
in particular those ones introduced in \cite{MP} for a
superlinear  $(p,q)$--equation, which allow us
to prove the existence of at least one nontrivial critical point of $\J$ 
in $X$, i.e., a weak bounded solution of \eqref{euler},
when the nonlinear term $g(x,t)$ is $(p-1)$--superlinear 
but does not satisfies the Ambrosetti--Rabinowitz condition. 

Problem \eqref{euler} with a $(p-1)$--superlinear
term $g(x,t)$ has been already studied if the Ambrosetti--Rabinowitz
condition, or a similar slightly more general assumption, holds 
(see \cite{AB1,CP1,CP2,CPS1,Ca}). Eventually, the term $A(x,t)|\xi|^p$ 
is replaced by some $\A(x,t,\xi)$, but both in \cite{AB1} and \cite{Ca}
it is assumed $A_t(x,t)t \ge 0$ a.e. in $\Omega$ for all $t\in \R$. 
On the contrary, in \cite{CP1,CP2,CPS1} such a product can also change sign
while, here, with the failure of the Ambrosetti--Rabinowitz condition,
we require $A_t(x,t)t \le 0$ (see Remark \ref{rema43}).
 
We note that, in order to find critical points of $\J$ in the intersection space $X$,
we cannot apply the classical Mountain Pass Theorem in \cite{AR}
as our functional $\J$ does not satisfy the Palais--Smale condition,
or its Cerami's
variant, in $X$ (Palais--Smale sequences 
may converge in $W^{1,p}_0(\Omega)$ and be unbounded in $L^\infty(\Omega)$, 
see, e.g., \cite[Example 4.3]{CP2017}). 
Hence, a weaker version of the Cerami--Palais--Smale 
condition is required and we can use a generalized 
version of the Mountain Pass Theorem (see Section \ref{secabstract}).

Since our main theorem covers very general situations
and a list of conditions is needed, we shall give 
the complete framework in Sections \ref{variational} and \ref{secmain}. 
However here, in order to highlight how our approach improves 
previous results, we consider the particular example
\begin{equation}\label{mod0}
A(x,t) \ =\ a(x) - \arctg|t|^{\theta} ,
\end{equation}
so that problem \eqref{euler} reduces to
\begin{equation}\label{euler0}
\left\{
\begin{array}{ll}
- {\rm div} \Big((a(x) -  \arctg |u|^{\theta}) |\nabla u|^{p-2}\nabla u\Big) \ -\dfrac{\theta}{p}
 \displaystyle \frac{|u|^{\theta-2} u}{1 + |u|^{2\theta}}  |\nabla u|^p\
 =\ g(x,u)
  &\hbox{in $\Omega$,}\\
u\ = \ 0 & \hbox{on $\partial\Omega$.}
\end{array}
\right.
\end{equation}

\begin{theorem}\label{coromain}
Let $a\in L^\infty(\Omega)$ be such that 
\[
a(x) \ge a_0 > \frac{\pi}{2}+\frac{\theta}{2p}\quad \hbox{a.e. in $\Omega$} 
\]
and assume that $1 < \theta\leq p$.
If $g(x,t)$ satisfies the assumptions $(G_0)$--$(G_4)$
stated in Sections $\ref{variational}$ and $\ref{secmain}$, 
for example, 
\begin{equation}\label{mod1}
g(x,t)=g_1(t)=\begin{cases} |t|^{q-2}t &\mbox{ if }|t|\leq 1\\
 |t|^{p-2}t\Big(\log|t|+1\Big) & \mbox{ if }|t|>1
\end{cases}\qquad \hbox{with $1 < p < q < p^\ast$,}
\end{equation}
then problem \eqref{euler0} admits at least one nontrivial bounded 
weak solution.
\end{theorem}

We note that the assumptions on $a(x)$ and $\theta$, given in Theorem \ref{coromain}, 
allow function $A(x,t)$ in \eqref{mod0} to verify all the conditions $(H_0)$--$(H_4)$
required in Section \ref{secmain}. 
Thus, Theorem \ref{coromain} is a corollary of Theorem \ref{main}
(see also Example \ref{exnoar}).


\section{Abstract tools} \label{secabstract}

Throughout this section, we denote $\N = \{1, 2, \dots\}$ and assume that:
\begin{itemize}
\item $(X, \|\cdot\|_X)$ is a Banach space with dual space
$(X',\|\cdot\|_{X'})$,
\item $(W,\|\cdot\|_W)$ is a Banach space such that
$X \hookrightarrow W$ continuously, i.e. $X \subset W$ and a constant $\rho_0 > 0$ exists
such that
\begin{equation}
\label{continuity}
\|u\|_W \ \le \ \rho_0\ \|u\|_X\qquad \hbox{for all $u \in X$,}
\end{equation}
\item $J : {\cal D} \subset W \to \R$ and $J \in C^1(X,\R)$ with $X \subset {\cal D}$.
\end{itemize}
Furthermore, fixing $c \in \R$, we define
\begin{itemize}
\item $K_J^c = \{u \in X:\ J(u) = c,\ dJ(u) = 0\}$
the set of the critical points of $J$ in $X$ at level $c$,
\item $J^c = \{u\in X:\ J(u) \le c\}$
the sublevel of $J$ with respect to $c$.
\end{itemize}

For simplicity, taking $c \in \R$, we say that a sequence
$(u_n)_n\subset X$ is a {\sl Cerami--Palais--Smale sequence at level $c$},
briefly {\sl $(CPS)_c$--sequence}, if
\[
\lim_{n \to +\infty}J(u_n) = c\quad\mbox{and}\quad 
\lim_{n \to +\infty}\|dJ(u_n)\|_{X'} (1 + \|u_n\|_X) = 0.
\]
Moreover, $c$ is a {\sl Cerami--Palais--Smale level}, briefly {\sl $(CPS)$--level}, 
if there exists a $(CPS)_c$--sequence.

Functional $J$ satisfies the classical Cerami--Palais--Smale condition in $X$ 
at level $c$ if every $(CPS)_c$--sequence converges in $X$
up to subsequences. Anyway, thinking about the setting of our problem,
in general $(CPS)_c$--sequences may also exist which are unbounded in $\|\cdot\|_X$
but converge with respect to $\|\cdot\|_W$. Then, we can weaken the classical Cerami--Palais--Smale 
condition in the following way.  

\begin{definition} \label{wCPSdef}
Given $c\in \R$, functional $J$ satisfies the
{\slshape weak Cerami--Palais--Smale 
condition at level $c$}, 
briefly {\sl $(wCPS)_c$ condition}, if for every $(CPS)_c$--sequence $(u_n)_n$,
a point $u \in X$ exists such that 
\begin{description}{}{}
\item[{\sl (i)}] $\displaystyle 
\lim_{n \to+\infty} \|u_n - u\|_W = 0\quad$ (up to subsequences),
\item[{\sl (ii)}] $J(u) = c$, $\; dJ(u) = 0$.
\end{description}
We say that $J$ satisfies the $(wCPS)$ condition in $I$, $I$ real interval,
if $J$ satisfies the $(wCPS)_c$ condition at each level $c \in I$.
\end{definition} 

Due to the convergence only in the norm $\|\cdot\|_W$, 
the $(wCPS)_c$ condition implies that the set of critical points of $J$ at level $c$
is compact with respect to $\|\cdot\|_W$. 
Anyway, this weaker ``compactness'' assumption allows one to prove 
the following Deformation Lemma  
(see \cite[Lemma 2.3]{CP3} which is stated in the weaker condition that
each $(CPS)$--level is a critical level, too). 

\begin{lemma}[Deformation Lemma] \label{deformo}
Let $J\in C^1(X,\R)$ and consider $c \in \R$ such that
\begin{itemize}
\item $J$ satisfies the $(wCPS)_c$ condition, 
\item $\ K_J^c = \emptyset$.
\end{itemize}
Then, fixing any $\bar \eps > 0$, there exist
a constant $\eps > 0$ and a homeomorphism $\psi : X \to X$
such that $2\eps < {\bar \eps}$ and
\begin{itemize}
\item[$(i)$] $\psi(J^{c+\eps}) \subset J^{c-\eps}$,
\item[$(ii)$] $\psi(u) = u$ for all $u \in X$ such that either
$J(u) \le c-\bar\eps$ or $J(u) \ge c+\bar\eps$.
\end{itemize}
Moreover, if $J$ is even on $X$, then $\psi$ can be chosen odd.
\end{lemma}

From Lemma \ref{deformo} the following generalization of the Mountain Pass Theorem
in \cite[Theorem 2.1]{AR} can be stated (for the proof, see \cite[Theorem 1.7]{CP3}).

\begin{theorem}[Mountain Pass Theorem]
\label{mountainpass}
Let $J\in C^1(X,\R)$ be such that $J(0) = 0$
and the $(wCPS)$ condition holds in $\R$.
Moreover, assume that two constants
$r_0$, $\varrho_0 > 0$ and a point $e \in X$ exist such that
\begin{equation}\label{aa1}
u \in X, \; \|u\|_W = r_0\quad \then\quad J(u) \ge \varrho_0,
\end{equation}
\begin{equation}\label{aa2}
\|e\|_W > r_0\qquad\hbox{and}\qquad J(e) < \varrho_0.
\end{equation}
Then, $J$ has a Mountain Pass critical point $u^* \in X$ such that $J(u^*) \ge \varrho_0$.
\end{theorem}


\section{Variational setting and first properties}
\label{variational}

Here and in the following, $|\cdot|$ is the standard norm 
on any Euclidean space as the dimension
of the considered vector is clear and no ambiguity arises
and $\meas(B)$ is the usual $N$--dimensional
Lebesgue measure of a measurable set $B$ in $\R^N$.
Furthermore, let $\Omega \subset \R^N$ be an open bounded domain, $N\ge 1$,
so we denote by:
\begin{itemize}
\item $L^\nu(\Omega)$ the Lebesgue space with
norm $|u|_\nu = \left(\int_{\Omega}|u|^\nu dx\right)^{1/\nu}$ if $1 \le \nu < +\infty$,  
 $u\in L^\nu(\Omega)$;
\item $L^\infty(\Omega)$ the space of Lebesgue--measurable 
and essentially bounded functions $u :\Omega \to \R$ with norm
\[
|u|_{\infty} = \esssup_{\Omega} |u|;
\]
\item $W_0^{1,p}(\Omega)$ the classical Sobolev space with
norm $\|u\|_{W} = |\nabla u|_p$ if $1\leq p<+\infty$, $u\in W_0^{1,p}(\Omega)$.
\end{itemize}

From now on, let $\, A : \Omega \times \R \to \R\,$
and $\, g :\Omega \times \R \to \R\,$ be such that
the following conditions hold:
\begin{itemize}
\item[$(H_0)$]
$A(x,t)$ is a $C^1$ Carath\'eodory function, i.e., \\
$A(\cdot,t) : x \in \Omega \mapsto A(x,t) \in \R$ is measurable for all $t \in \R$,\\
$A(x,\cdot) : t \in \R \mapsto A(x,t) \in \R$ 
is $C^1$ for a.e. $x \in \Omega$ with $A_t(x,t) = \frac{\partial}{\partial t} A(x,t)$;
\item[$(H_1)$] $A(x,t)$ and $A_t(x,t)$ are essentially bounded
if $t$ is bounded, i.e., 
\[
\sup_{|t| \le r} |A(\cdot,t)| \in L^\infty(\Omega),\quad 
\sup_{|t| \le r} |A_t(\cdot,t)| \in L^\infty(\Omega)
\qquad \hbox{for any $r > 0$;}
\]
\item[$(G_0)$] 
$g(x,t)$ is a Carath\'eodory function, i.e.,\\
$g(\cdot,t) : x \in \Omega \mapsto g(x,t) \in \R$ is measurable for all $t \in \R$,\\
$g(x,\cdot) : t \in \R \mapsto g(x,t) \in \R$ is continuous for a.e. $x \in \Omega$;
\item[$(G_1)$] $a_1$, $a_2 > 0$ and $q \ge 1$ exist such that
\[
|g(x,t)| \le a_1 + a_2 |t|^{q-1} \qquad
\hbox{a.e. in $\Omega$, for all $t \in \R$.}
\]
\end{itemize}

\begin{remark}\label{suG}
By definition, it is $G(x,0) = 0$ a.e. in $\Omega$; furthermore,
from $(G_0)$--$(G_1)$ it follows that $G(x,t)$ is a $C^1$ Carath\'eodory function
in $\Omega \times \R$ and there exist $a_3$, $a_4 > 0$ such that
\begin{equation}
\label{alto3}
|G(x,t)| \le a_3 + a_4 |t|^q\qquad\hbox{a.e in $\Omega$, for all $t \in \R$.}
\end{equation}
We note that, unlike the classical assumption $(G_1)$ which requires 
$q<p^*$ for obtaining the regularity of the associated Nemytskii operator
(see \cite{AP}), here no upper bound on $q$ is actually assumed. 
\end{remark}

In order to investigate the existence of weak solutions  
of the nonlinear problem \eqref{euler}, the notation introduced for the abstract 
setting at the beginning of Section \ref{secabstract}
is referred to our problem with $W = W^{1,p}_0(\Omega)$ and the
Banach space $(X,\|\cdot\|_X)$ defined as
\begin{equation}\label{space}
X := W^{1,p}_0(\Omega) \cap L^\infty(\Omega),\qquad
\|u\|_X = \|u\|_W + |u|_\infty.
\end{equation}
Moreover, from the Sobolev Embedding Theorem, for any $\nu \in [1,p^*[$,
with $p^* = \frac{pN}{N-p}$ as $N > p$ otherwise $p^*=+\infty$,
a constant $\rho_\nu > 0$ exists, such that 
\[
|u|_\nu\ \le\ \rho_\nu \|u\|_W \quad \hbox{for all $u \in W^{1,p}_0(\Omega)$}
\]
and the embedding $W^{1,p}_0(\Omega) \hookrightarrow\hookrightarrow L^\nu(\Omega)$
is compact.

From the definition of $X$, we have that $X \hookrightarrow W^{1,p}_0(\Omega)$ and $X \hookrightarrow L^\infty(\Omega)$
with continuous embeddings, and \eqref{continuity} holds with $\rho_0 = 1$.

We note that $X = W^{1,p}_0(\Omega)$ if $p > N>1$ or $p\geq N=1$,  as in 
these cases $W^{1,p}_0(\Omega) \hookrightarrow L^\infty(\Omega)$, so the 
abstract part is the standard one with the usual Mountain Pass Theorem.

Now, we consider the functional $\J : X \to \R$ defined as \eqref{funct}.

Taking any $u$, $v\in X$, by direct computations
it follows that its G\^ateaux differential in $u$ along the direction $v$ is
\begin{equation}
\label{diff}
\langle d\J(u),v\rangle = \int_\Omega A(x,u) |\nabla u|^{p-2} \nabla u\cdot \nabla v\ dx
+ \frac1p\ \int_{\Omega} A_t(x,u) v |\nabla u|^{p} dx
- \int_\Omega g(x,u)v\ dx .
\end{equation}

The following regularity result holds (see \cite[Proposition 3.2]{CPS1}). 

\begin{proposition}\label{smooth1}
Taking $p > 1$, assume that $(H_0)$--$(H_1)$, $(G_0)$--$(G_1)$ hold.
If $(u_n)_n \subset X$, $u \in X$ are such that
\[
\|u_n - u\|_W \to 0, \quad u_n \to u\; \hbox{a.e. in $\Omega$} \ \qquad\hbox{if $n \to+\infty$}
\]
\[
\hbox{and $M > 0$ exists so that $|u_n|_\infty \le M$ for all $n \in \N$,}
\]
then
\[
\J(u_n) \to \J(u)\quad \hbox{and}\quad \|d\J(u_n) - d\J(u)\|_{X'} \to 0
\quad\hbox{if $\ n\to+\infty$.}
\]
Hence, $\J$ is a $C^1$ functional on $X$ with Fr\'echet differential  
defined as in \eqref{diff}.
\end{proposition}


\section{Statement of the main result} \label{secmain}

From now on, we assume that in addition to hypotheses $(H_0)$--$(H_1)$ and $(G_0)$--$(G_1)$,
functions $A(x,t)$ and $g(x,t)$ satisfy the following further conditions:
\begin{itemize}
\item[$(H_2)$] a constant $\alpha_0 > 0$ exists such that
\[
A(x,t) \ge \alpha_0 \qquad \hbox{a.e. in $\Omega$, for all $t \in \R$;}
\]
\item[$(H_3)$] some constants $R_0 \ge 1$ and $\alpha_1 > 0$ exist such that
\[
A(x,t) + \frac1p A_t(x,t) t \ge \alpha_1 A(x,t)\quad\hbox{a.e. in $\Omega$ if $|t| \ge R_0$;}
\]
\item[$(H_4)$] $\quad A_t(x,s t) s^{p+1} t \ge A_t(x,t) t \; $ for all $s\in [0,1]$,
for a.e. $x \in \Omega$ and all $t \in \R$;
\item[$(G_2)$] $\; \displaystyle \lim_{|t| \to +\infty} \frac{G(x,t)}{|t|^{p}}\ =\ +\infty\;$ 
uniformly for a.e. $x \in \Omega$;
\item[$(G_3)$] taking $\sigma(x,t) = g(x,t) t - p G(x,t)$, assume that $\beta \in L^1(\Omega)$ exists
such that $\beta(x) \ge 0$ a.e. in $\Omega$ and 
\[
\sigma(x,t_1) \le \sigma(x,t_2) + \beta(x)\quad\hbox{a.e. in $\Omega$, 
for all $0 \le t_1 \le t_2$ or $t_2\le t_1 \le 0$;}
\]
\item[$(G_4)$] $\displaystyle \lim_{t\to 0}\frac{G(x,t)}{|t|^p}=0\, $ uniformly for a.e. $x\in \Omega$. 
\end{itemize}

\begin{remark}
We emphasize the fact that by $(H_3)$ we can handle the case $A_t(x,t)t\leq 0$. Notice that, otherwise, condition $(H_3)$ can be omitted when $g$ satisfies the Ambrosetti-Rabinowitz condition, see \cite{pellacci}. In this way, our existence result also extends the one proved in \cite{pellacci} in the difficult situation in which g does not satisfies the Ambrosetti-Rabinowitz condition.
\end{remark}

\begin{remark}
Condition $(G_3)$ was introduced in \cite{MP} in order to prevent the 
use of the Ambrosetti--Rabinowitz condition, and a slight improvement has been 
recently proposed in \cite{mupli}. 
See also \cite{bernimug} for an application in a different framework.
\end{remark}

\begin{example}\label{exnoar}
We note that function $g_1(t)$ in \eqref{mod1} fails to satisfy the Ambrosetti--Rabinowitz condition
but verifies conditions $(G_0)$--$(G_4)$.
On the contrary, function
\[
g(x,t) = g_2(t) = |t|^{q-2}t \qquad \hbox{with $p<q<p^\ast$,}
\]
satisfies both the Ambrosetti-Rabinowitz condition and hypotheses $(G_0)$--$(G_4)$.
\end{example}

Our main result reads as follows.

\begin{theorem}\label{main}
Assume $(H_0)$--$(H_4)$ and $(G_0)$--$(G_4)$.
Then problem \eqref{euler} admits a nontrivial bounded weak solution.
\end{theorem}

\begin{remark}\label{rema43}
Taking $s=0$ in $(H_4)$ we have that
\begin{equation}
\label{altoAt}
A_t(x,t) t \le 0\quad \hbox{a.e. in $\Omega$, for all $t \in \R$.}
\end{equation}
Hence, from $(H_0)$ and \eqref{altoAt} it follows that for a.e. $x \in \Omega$ the $C^1$ map
$A(x,\cdot)$ is increasing in $]-\infty,0]$, decreasing in $[0,+\infty[$, then
it attains its maximum in $t=0$. On the other hand, $(H_1)$ implies that 
$A(\cdot,0) \in L^\infty(\Omega)$; hence, $\gamma_A >0$
exists such that 
\begin{equation}
\label{altoA}
A(x,t) \le \gamma_A\qquad\hbox{a.e. in $\Omega$, for all $t \in \R$}.
\end{equation}
Such a requirement was already assumed in \cite{AB1} and \cite{Ca}.
\end{remark}

\begin{remark}
From $(G_0)$--$(G_2)$ and direct computations it follows that 
for all $\mu > 0$ a constant $L_\mu > 0$ exists, such that
\begin{equation}
\label{altomu}
G(x,t) \ge \mu |t|^p - L_\mu\qquad\hbox{a.e in $\Omega$, for all $t \in \R$.}
\end{equation}
We note that, for the arbitrariness of $\mu$, 
\eqref{altomu} and $(G_1)$ imply $p < q$.
\end{remark}

\begin{remark}
Condition $(G_3)$ implies that
$\sigma(x,0) = 0$ a.e in $\Omega$,
and then
\[
\sigma(x,t) \ge -\beta(x)\qquad\hbox{a.e in $\Omega$, for all $t \in \R$.}
\]
Hence,
\begin{equation}
\label{bassosigma}
\int_\Omega \sigma(x,u)dx\ \ge\ - |\beta|_1\qquad\hbox{for all $u \in X$.}
\end{equation}
\end{remark}


\section{Proof of the main result} \label{secproof}

The goal of this section is to prove the existence of a weak
bounded nontrivial solution of problem \eqref{euler},
so, by using the variational principle which follows 
from Proposition \ref{smooth1}, we want 
to apply Theorem \ref{mountainpass} to the functional $\J$ in \eqref{funct}
on the Banach space $X$ as in \eqref{space}.

\begin{proposition}\label{wCPS}
If $1 < p < q < p^*$ and $(H_0)$--$(H_4)$, $(G_0)$--$(G_3)$ hold, then 
functional $\J$ satisfies the weak Cerami--Palais--Smale condition in $X$
at each level $c \in \R$.
\end{proposition}

\begin{proof}
Let $c \in \R$ be fixed and consider a sequence $(u_n)_n \subset X$
such that
\begin{equation}\label{c1}
\J(u_n) = c +\eps_n\quad \hbox{and}\quad \|d\J(u_n)\|_{X'}(1 + \|u_n\|_X)=\eps_n,
\end{equation}
where, for simplicity, throughout this proof, we use the notation $(\eps_n)_n$
for any infinitesimal sequence depending only on $(u_n)_n$. 

Firstly, we want to prove that
\begin{equation}\label{c0}
(u_n)_n\quad \hbox{is bounded in $W_0^{1,p}(\Omega)$}.
\end{equation} 
The ideas of the proof of \eqref{c0} are essentially 
contained in \cite[Lemma 2.2]{FL} and \cite[Proposition 3]{MP}, see also \cite[Lemma 2.5]{Liu},
but since some changes are required we include here all the details for the reader's convenience.

To this aim, arguing by contradiction, we assume that
\begin{equation}\label{c2}
\|u_n\|_W \to +\infty\qquad
\mbox{if $\ n\to+\infty$}
\end{equation}
and for any $n \in \N$ we define
\begin{equation}\label{c3}
v_n (x) = \frac{u_n(x)}{\|u_n\|_W} \quad
\mbox{for a.e. $x \in \Omega$,}
\end{equation}
so that $v_n\in X$.
Since $(v_n)_n$ is bounded in $W_0^{1,p}(\Omega)$, a function $v \in W_0^{1,p}(\Omega)$ exists
such that, up to subsequences, 
\begin{eqnarray}
&&v_n \rightharpoonup v\; \hbox{weakly in $W_0^{1,p}(\Omega)$,}
\nonumber\\
&&v_n \to v\; \hbox{strongly in $L^\nu(\Omega)$ for each $\nu \in [1,p^*[$,}
\label{c5}\\
&&v_n \to v\; \hbox{a.e. in $\Omega$.}
\label{c6}
\end{eqnarray}
Assume that $v \not\equiv 0$ in $\Omega$, i.e.,
\begin{equation}\label{c7}
\meas(\Omega \setminus \Omega_0) > 0,\qquad \hbox{with $\Omega_0 = \{x \in \Omega: v(x) = 0\}$.}
\end{equation}
From definition \eqref{c3}, the definition in \eqref{c7} and \eqref{c2}, \eqref{c6} it follows that
\[
|u_n(x)| = |v_n(x)|\ \|u_n\|_W \ \to +\infty\quad \hbox{for a.e. $x\in \Omega \setminus \Omega_0$;}
\]
hence, $(G_3)$ and \eqref{c6} imply that
\[ 
\frac{G(x,u_n(x))}{\|u_n\|_W^{p}}\ =\
\frac{G(x,u_n(x))}{|u_n(x)|^{p}}\ |v_n(x)|^{p} \ \to\ +\infty\quad
\hbox{for a.e. $x\in \Omega \setminus \Omega_0$.}
\]
Thus, from Fatou's Lemma and \eqref{c7} it follows that
\[
\int_{\Omega \setminus \Omega_0} \frac{G(x,u_n)}{\|u_n\|_W^{p}} dx\ \to\ +\infty,
\]
which implies that
\begin{equation}\label{c8}
\int_{\Omega} \frac{G(x,u_n)}{\|u_n\|_W^{p}} dx\ \to\ +\infty,
\end{equation}
as from \eqref{altomu} with, e.g., $\mu =1$, and \eqref{c2}
we obtain that 
\[
\int_{\Omega_0} \frac{G(x,u_n)}{\|u_n\|_W^{p}} dx\ \ge\ - \frac{L_1 \meas(\Omega_0)}{\|u_n\|_W^{p}} = \eps_n.
\]
But \eqref{funct}, \eqref{c1}, \eqref{c3} and \eqref{altoA} imply that
\[
\eps_n\ =\ -  \frac{\J(u_n)}{\|u_n\|_W^{p}}
\ =	\ - \frac{\gamma_A}p + \int_{\Omega} \frac{G(x,u_n)}{\|u_n\|_W^{p}} dx
\]
which contradicts \eqref{c8}. Hence, \eqref{c7} cannot hold and it has to be $v(x)=0$
a.e. in $\Omega$.

Now, from Proposition \ref{smooth1} we have that the map
\[
s \in [0,1] \mapsto \J(s u_n) \in \R
\]
is $C^1$ in its domain for each $n \in \N$; then
$s_n \in [0,1]$ exists such that
\begin{equation}\label{c9}
\J(s_n u_n)\ =\ \max_{s \in [0,1]} \J(s u_n).  
\end{equation}
If we fix any $\lambda >0$ and define 
\[
w_n(x)\ =\ (2\lambda)^{\frac1p} v_n(x)\quad \mbox{for a.e. $x \in \Omega$,}
\]
we have that $w_n\in X$; moreover, from \eqref{c5} and \eqref{c6}
 it follows that
\[
\begin{split}
&w_n \to 0\; \hbox{strongly in $L^\nu(\Omega)$ for each $\nu \in [1,p^*[$,}\\
&w_n \to 0\; \hbox{a.e. in $\Omega$.}
\end{split}
\]
Hence, from Remark \ref{suG} with $q < p^*$, by using the continuity of the 
Nemytskii operator, we obtain that 
\[
\int_{\Omega} G(x,w_n) dx\ \to\ 0;
\]
thus, $n_1= n_1(\lambda) \in \N$ exists, such that
\begin{equation}\label{c10}
\left|\int_{\Omega} G(x,w_n) dx\right|\ < \ \frac{\lambda \alpha_0}{p}\qquad
\hbox{for all $n \ge n_1$,}
\end{equation}
with $\alpha_0$ as in $(H_2)$. 
We note that \eqref{c2} implies 
\[
\frac{(2\lambda)^{\frac1p}}{\|u_n\|_W} \ \to\ 0,
\] 
so $n_2  = n_2(\lambda)\ge n_1$ exists, such that
\[
0\ <\ \frac{(2\lambda)^{\frac1p}}{\|u_n\|_W} \ < 1\qquad
\mbox{for all $n \ge n_2$};
\]
then from \eqref{c9}, $(H_2)$, \eqref{c10}
and direct computations it follows that 
\[
\J(s_n u_n) \ \ge\ \J(w_n)\ \ge
\frac{2 \lambda \alpha_0}{p}\ -\ \int_{\Omega} G(x,w_n) dx\ \ge\  \frac{\lambda \alpha_0}{p}\quad
\hbox{for all $n \ge n_2$.}
\]
Whence, as $\lambda >0$ is arbitrary, we obtain that
\begin{equation}\label{c17}
\J(s_n u_n) \ \to\ +\infty\qquad
\hbox{if $n \to +\infty$.}
\end{equation}
As $\J(0) = 0$, from \eqref{c1},  the limit \eqref{c17}
implies that $n_0 \in \N$ exists such that for all $n \ge n_0$
it has to be $s_n \in ]0,1[$ and then, 
from the Fermat Theorem, we have that
\[
\frac{d}{ds}\J(su_n)|_{s=s_n}\ = \ 0 \qquad \hbox{for all $n \ge n_0$,}
\]
which implies
\[
\begin{split}
0\ &=\ 
s_n \frac{d}{ds}\J(su_n)|_{s=s_n}\ = \ \langle d\J(s_n u_n),s_n u_n\rangle\\
&=\ \int_{\Omega} A(x,s_n u_n) |\nabla (s_n u_n)|^p dx 
+ \frac{1}p\ \int_{\Omega} A_t(x,s_nu_n) s_n^{p+1} u_n |\nabla u_n|^p dx 
- \int_{\Omega} g(x,s_nu_n)s_nu_n dx,
\end{split}
\] 
i.e.,
\begin{equation}\label{c11}
\begin{split}
\int_{\Omega} A(x,s_n u_n) |\nabla (s_n u_n)|^p dx\ =\
&-\ \frac{1}p\ \int_{\Omega} A_t(x,s_nu_n) s_n^{p+1} u_n |\nabla u_n|^p dx \\
&+ \int_{\Omega} g(x,s_nu_n)s_nu_n dx.
\end{split}
\end{equation}
Now, from one hand, we note that \eqref{funct}, \eqref{diff}, \eqref{c1} 
and \eqref{altoAt}, \eqref{bassosigma}, imply that
\[
\begin{split}
p c + \eps_n\ =\ & \ p \J(u_n) - \langle d\J(u_n),u_n\rangle \\
=\ &- \frac{1}p\ \int_{\Omega} A_t(x,u_n) u_n |\nabla u_n|^p dx + \int_{\Omega} \sigma(x,u_n) dx
\ \ge\ - |\beta|_1; 
\end{split}
\]
hence, 
\begin{equation}\label{c18}
\left(- \frac{1}p\ 
\int_{\Omega} A_t(x,u_n) u_n |\nabla u_n|^p dx + \int_{\Omega} \sigma(x,u_n) dx\right)_n\qquad
\mbox{is bounded;}
\end{equation}
while, on the other hand, $s_n \in [0,1]$ and $(G_3)$ give
\[
\sigma(x,s_nu_n(x)) \le \sigma(x,u_n(x)) + \beta(x)\quad\hbox{for a.e. $x\in \Omega$, 
all $n \in \N$;}
\]
thus,
\begin{equation}\label{c12}
\int_{\Omega} \sigma(x,s_nu_n) dx\ \le\ 
\int_{\Omega} \sigma(x,u_n) dx + |\beta|_1\quad\hbox{for all $n \in \N$.}
\end{equation}
Summing up, from definition \eqref{funct}, estimates 
\eqref{c11}, \eqref{c12}, assumption $(H_4)$ and \eqref{c18}, 
for all $n \ge n_0$ we obtain that
\[
\begin{split}
\J(s_nu_n) \ &=\ -\ \frac{1}{p^2}\ \int_{\Omega} A_t(x,s_nu_n) s_n^{p+1} u_n |\nabla u_n|^p dx 
+ \frac{1}{p} \int_{\Omega} \sigma(x,s_nu_n) dx\\
&\le\ -\ \frac{1}{p^2}\ \int_{\Omega} A_t(x,u_n) u_n |\nabla u_n|^p dx 
+ \frac{1}{p} \int_{\Omega} \sigma(x,u_n) dx + |\beta|_1 \ \le \ b
\end{split}
\]
for some $b>0$, in contradiction with \eqref{c17}. In conclusion, \eqref{c0} is true and $u \in W^{1,p}_0(\Omega)$
exists such that, up to subsequences, we have
\[
\begin{split}
&u_n \rightharpoonup u\; \hbox{weakly in $W^{1,p}_0(\Omega)$,}\\
&u_n \to u\; \hbox{strongly in $L^\nu(\Omega)$ for each $\nu \in [1,p^*[$,}\\
&u_n \to u\; \hbox{a.e. in $\Omega$.}
\end{split}
\]
Now, proceeding exactly as in Steps 2--5 of the proof of \cite[Proposition 4.6]{CP2}, 
it has to be that $u \in L^\infty(\Omega)$, too, and not only
$u_n \to u$ strongly in $W^{1,p}_0(\Omega)$ but also
$u$ is a critical point of $\J$ in $X$ such that $\J(u)=c$.
\end{proof}

\begin{proof}[Proof of Theorem $\ref{main}$.] 
From \eqref{alto3}, $(G_4)$ and direct computations
we get that for every $\eps>0$ 
a constant $C_\eps>0$ exists such that
\begin{equation}\label{suGg}
G(x,t)\leq \eps |t|^p +C_\eps |t|^q\quad \mbox{for a.e. $x\in \Omega$ and for all $t\in \R$.}
\end{equation}
Then, from \eqref{funct}, \eqref{suGg}, $(H_2)$, the Poincar\'e and the Sobolev inequalities it follows that
\[
\begin{split}
\J(u)\ &\geq\ \frac{1}{p}\int_\Omega A(x,u)|\nabla u|^p dx-\eps
 \int_\Omega|u|^pdx-C_\eps \int_\Omega |u|^q dx\\
&\geq\ \left(\frac{\alpha_0}{p} - \frac{\eps}{\lambda_1}\right) \|u\|_W^p 
- \tilde C_\eps \|u\|_W^q
\end{split}
\]
for some $\tilde C_\eps$ and for all $u\in X$, where
\[
\lambda_1\ =\ \inf_{\stackrel[u\neq0]{}{u\in W^{1,p}_0(\Omega)}}
\frac{\displaystyle\int_\Omega |\nabla u|^pdx}{\displaystyle\int_\Omega |u|^pdx}.
\]
Hence, if $\eps<\frac{\lambda_1\alpha_0}{p}$ and $\|u\|_W$ is small enough, 
we immediately deduce that 0 is a local minimum point for $\J$ and \eqref{aa1}
in Proposition \ref{mountainpass} holds for suitable $r_0$, $\varrho_0 >0$.

On the other hand, denoting by $\varphi_1$ the first positive eigenfunction of 
$-\Delta_p$ in $W^{1,p}_0(\Omega)$ with $|\varphi_1|_p=1$, from \eqref{altomu} 
with any fixed $\mu>0$, and from \eqref{funct} and \eqref{altoA} we get 
\[
\J(s\varphi_1)\ \leq\ s^p\left(\frac{\gamma_A}{p} \lambda_1
-\mu\right) + L_\mu|\Omega| \quad \hbox{for any $s>0$.}
\]
Hence, by choosing $\mu$ and $s$ sufficiently large, 
we obtain that $\J$ satisfies also the geometrical assumption \eqref{aa2} of 
Theorem \ref{mountainpass}; thus, 
by Proposition \ref{wCPS} we can apply Theorem \ref{mountainpass} 
and conclude with the existence of a nontrivial solution to problem \eqref{euler}.
\end{proof}

\section*{Acknowledgements}
A.M. Candela and G. Fragnelli are partially supported by  
MIUR--PRIN project  ``Qualitative and quantitative aspects of nonlinear PDEs'' (2017JPCAPN\underline{\ }005)
and by Fondi di Ricerca di Ateneo 2015/16. 
G. Fragnelli acknowledges the support of FFABR ``Fondo per il finanziamento
 delle attivit\`a base di ricerca'' 2017 and of the INdAM--GNAMPA Project 2019 
``Controllabilit\`a di PDE in modelli fisici e in scienze della vita''.
D. Mugnai is partially supported by MIUR--PRIN  ``Variational methods,
with applications to problems in mathematical physics and geometry'' (2015KB9WPT\underline{\ }009) 
and by the FFABR ``Fondo per il finanziamento delle attivit\`a base di ricerca'' 2017.

The authors wish to thank one of the two anonymous referees for bringing \cite{pellacci} to their attention.


\begin{thebibliography}{99}

\bibitem{AP}
A. Ambrosetti and G. Prodi, 
{\em A Primer of Nonlinear Analysis},
Cambridge University Press, Cambridge, 1993. 

\bibitem{AR} 
A. Ambrosetti and P.H. Rabinowitz, 
Dual variational methods in critical point theory and applications, 
{\em J. Funct. Anal.} {\bf 14} (1973), 349--381.

\bibitem{AB1} D. Arcoya and L. Boccardo, Critical points
for multiple integrals of the calculus of variations, {\em Arch.
Rational Mech. Anal.} {\bf 134} (1996), 249--274.

\bibitem{bernimug}
F. Bernini and D. Mugnai, 
On a logarithmic Hartree equation, 
\emph{Adv. Nonlinear Anal.}, doi.org/10.1515/anona-2020-0028.

\bibitem{CP1} A.M. Candela and G. Palmieri, {Multiple solutions
of some nonlinear variational problems}, \emph{Adv. Nonlinear Stud.} {\bf 6} (2006), 269--286.

\bibitem{CP2}
A.M. Candela and G. Palmieri,
{Infinitely many solutions of some nonlinear variational equations},
\emph{Calc. Var. Partial Differential Equations} {\bf 34} (2009), 495--530.

\bibitem{CP3} A.M. Candela and G. Palmieri, {Some abstract critical point theorems
and applications}. In: \emph{Dynamical Systems, Differential Equations and Applications} 
(X. Hou, X. Lu, A. Miranville, J. Su \& J. Zhu Eds), 
\emph{Discrete Contin. Dynam. Syst.} \textbf{Suppl. 2009} (2009), 133--142. 

\bibitem{CP2017}
A.M. Candela and G. Palmieri,
Multiplicity results for some nonlinear elliptic problems
with asymptotically $p$--linear terms,
\emph{Calc. Var. Partial Differential Equations} \textbf{56}:72 (2017) (39 pages). 

\bibitem{CPS1} 
A.M. Candela, G. Palmieri and A. Salvatore,
Some results on supercritical quasilinear elliptic problems (submitted).

\bibitem{Ca} A. Canino, Multiplicity of solutions for quasilinear elliptic equations,
{\em Topol. Methods Nonlinear Anal.} {\bf 6} (1995), 357--370.

\bibitem{CD}
A. Canino and M. Degiovanni, Nonsmooth critical point theory and quasilinear elliptic
equations, in: ``\emph{Topological Methods in Differential Equations and Inclusions}'' 1-50
(A. Granas, M. Frigon \& G. Sabidussi Eds),  
{\em NATO Adv. Sci. Inst. Ser. C Math. Phys. Sci.} \textbf{472}, Kluwer Acad. Publ., Dordrecht, 1995. 

\bibitem{CDM} 
J.N. Corvellec, M. Degiovanni and M. Marzocchi, Deformation properties for continuous
functionals and critical point theory, {\em Topol. Methods Nonlinear Anal.}
\textbf{1} (1993), 151--171.

\bibitem{FL}
F. Fang and S. Liu, 
Nontrivial solutions of superlinear $p$--Laplacian equations,
\emph{J. Math. Anal. Appl.}
\textbf{351} (2009), 138--146.

\bibitem{Liu}
S. Liu, On superlinear problems without Ambrosetti and
Rabinowitz condition, \emph{Nonlinear Anal.}
\textbf{73} (2010), 788--795.

\bibitem{MP}
D. Mugnai and N.S. Papageorgiou,
Wang's multiplicity result for superlinear
$(p,q)$--equations without the Ambrosetti--Rabinowitz
condition, \emph{Trans. Amer. Math. Soc.} \textbf{366} (2014), 4919--4937.

\bibitem{mupli}
D. Mugnai and E. Proietti Lippi, 
Neumann fractional $p$--Laplacian: eigenvalues and existence results, 
\emph{Nonlinear Anal.} \textbf{188} (2019), 455--474.

\bibitem{pellacci}
B. Pellacci, Critical points for some functionals of the calculus of variations, \emph{Topol. Methods
Nonlinear Anal.} {\bf17} (2001), 285--305.


\end{thebibliography}
\end{document}